\documentclass[a4paper,reqno]{amsart}
\usepackage[text={5.2in,7.8in},centering]{geometry}
\usepackage[initials]{amsrefs}
\usepackage{amssymb}
\usepackage{enumerate}
\usepackage{mathtools}

\numberwithin{equation}{section}
\theoremstyle{plain}
\newtheorem{theorem}{Theorem}[section]
\newtheorem{proposition}[theorem]{Proposition}

\newtheorem{lemma}[theorem]{Lemma}
\theoremstyle{remark}

\newtheorem*{remark*}{Remark}
\newtheorem*{remarks*}{Remarks}
\theoremstyle{definition}

\newtheorem*{assumptionone*}{Assumption (H1)}
\newtheorem*{assumptiontwo*}{Assumption (H2)}
\newtheorem*{definition*}{Definition}
\newtheorem*{notation*}{Notation}
\newtheorem*{notations*}{Notations}

\providecommand{\BS}[1]{\boldsymbol{#1}}
\providecommand{\C}[1]{\mathcal{#1}}
\providecommand{\D}[1]{\mathbb{#1}}

\newcommand{\dd}{\mathrm{d}}
\newcommand{\eul}{\mathrm{e}}
\newcommand{\ii}{\mathrm{i}}
\DeclarePairedDelimiter\abs{\lvert}{\rvert}
\DeclarePairedDelimiter\accol{\lbrace}{\rbrace}
\DeclarePairedDelimiter\croch{\lbrack}{\rbrack}
\DeclarePairedDelimiter\norm{\lVert}{\rVert}
\DeclarePairedDelimiter\scal{\langle}{\rangle}

\DeclareMathOperator{\ord}{O}

\DeclareMathOperator{\vece}{e}
\begin{document}
\title[Asymptotics of large eigenvalues for a class of band matrices]
{Asymptotics of large eigenvalues\\ for a class of band matrices}
\author[A.~Boutet de Monvel]{Anne Boutet de Monvel$^1$}
\author[J.~Janas]{Jan Janas$^2$}
\author[L.~Zielinski]{Lech Zielinski$^3$}
\address{$^1$Institut de Math\'ematiques de Jussieu\\ 
Universit\'e Paris Diderot Paris 7\\ 
175 rue du Chevaleret, 75013 Paris, France\\ 
E-mail: aboutet@math.jussieu.fr}
\address{$^2$Instytut Matematyczny PAN\\ 
ul.~Sw.~Tomasza 30, 31-027 Krak\'ow, Poland\\ 
E-mail: najanas@cyf-kr.edu.pl} 
\address{$^3$LMPA, Universit\'e du Littoral\\ 
50 Rue F. Buisson, B.P. 699, 62228 Calais, France\\
E-mail: Lech.Zielinski@lmpa.univ-littoral.fr}
\subjclass[2010]{Primary: 47B36; Secondary: 47B25, 47A75, 47A10, 47A55}
\keywords{band matrix, Jacobi matrix, eigenvalues, asymptotics}
\date{}
\begin{abstract}
We investigate the asymptotic behaviour of large eigenvalues for a class of finite difference self-adjoint operators with compact resolvent in $l^2$.
\end{abstract} 
\maketitle
\section{Introduction} \label{sec:intro}
\subsection{General remarks} \label{sec:11}
Infinite tridiagonal tridiagonal matrices called ``Jacobi matrices'' have been investigated in many recent papers in relation with various questions of pure and applied mathematics (see \cites{BNS,JN,Tes,Tur,V}). In \cites{BNS,BZ,JN,M1,Z} the authors investigate Jacobi matrices acting in $l^2$ as unbounded self-adjoint operators with discrete spectrum and asymptotic formulas for large eigenvalues are given. This type of analysis is of particular interest in Quantum Physics when information about physical parameters can be deduced from the spectral asymptotics of a concrete model. 

Except for recent work \cite{M2} there are no corresponding work concerning the asymptotic analysis of large eigenvalues of higher order symmetric difference operators. This fact is not so surprising because higher order difference operators had not been studied from the spectral point of view, up to the last few years. On the other hand there are already works dealing with spectral properties of difference operators of higher order, see, for example \cites{BN1,BN2,Sh}. It is natural to ask if known results on tridiagonal matrices can be generalized to higher order difference operators and to look for applications. As an example of possible application let us mention here the problem of the behavior of large singular values for a non-symmetric Jacobi matrix $J$ with discrete spectrum. Indeed, these singular values are eigenvalues of $J^*J$ which is a symmetric difference operator of order four.

The aim of this paper is to obtain a simple remainder estimate in the asymptotics of eigenvalues for a large class of symmetric higher order difference operators. Applying this result to tridiagonal matrices,
\begin{enumerate}[(i)]
\item
we obtain asymptotic estimates of eigenvalues for Jacobi matrices which cannot be treated in \cites{BNS,BZ,JN,M1} (see Section~\ref{sec:13}),
\item
but with remainder estimates less precise than those of \cites{BNS,BZ,JN,M1}.
\end{enumerate} 
Since our assumptions are weaker than in \cites{BNS,BZ,JN,M1,M2}, we must overcome some additional difficulties, but the main idea of our approach remains the same as in \cite{BZ}. Although the approach of this paper is used to obtain the simplest remainder estimate, it is possible to follow the idea of \cite{BZ} in order to compute further terms of the asymptotics with smaller remainders under stronger ``conditions of smoothness'' imposed on the entries. It is an open question how to extend the methods of the present work to the case of non-smooth entries. Finally notice that we (and most of the above authors) have not studied asymptotic properties at infinity of the eigenvectors of $A$.
\subsection{Main statement} \label{sec:12}
Let $l^2=l^2(\D{N}^*)$ be the Hilbert space of square summable complex valued sequences $x\colon\D{N}^*\to\D{C}$ equipped with the norm
\[
\norm{x}_{l^2}=\biggl(\,\sum_{j=1}^{\infty}\abs{x(j)}^2\biggr)^{\! 1/2}.
\]
We fix $d\colon\D{N}^*\to\D{R}$ and introduce
\begin{equation} \label{11}
\C{D}\coloneqq\Bigl\{x\in l^2\colon\sum_{j=1}^{\infty} |d(j)x(j)|^2 <\infty\Bigr\}.
\end{equation}
Then we consider the self-adjoint operator $D\colon\C{D}\to l^2$ given by
the formula
\begin{equation} \label{12}
(Dx)(j)=d(j)x(j)\text{ for }x\in\C{D}
\end{equation}
and a finite difference operator $A'\colon\C{D}\to l^2$ of the form
\begin{equation} \label{13}
(A'x)(j)=\sum_{1\leq l\leq m}\left(a_l(j)x(j+l)+a_l(j-l)x(j-l)\right)
\end{equation}
where the coefficients $a_l\colon\D{N}^*\to\D{R}$, $l=1,\dots,m$ satisfy
\begin{equation} \label{14}
\frac{|a_l(j-l)|+|a_l(j)|}{d(j)}\xrightarrow[j\to\infty]{}0.
\end{equation}
We assume  $a_l(j-l)=0=x(j-l)$ when $j\leq l$ in \eqref{13}. We investigate the operator
\begin{equation} \label{15}
A=D+A'
\end{equation}
under the additional assumption
\begin{equation} \label{16}
d(n)\xrightarrow[n\to\infty]{}\infty,
\end{equation}
which ensures that $A$ has compact resolvent, hence there exists an orthonormal basis $(v_n)_{n\in\D{N}^*}$ such that $Av_n=\lambda_n(A)v_n$ holds for $n\in\D{N}^*$, $\lambda_n(A)\to\infty$ as $n\to\infty$ and $(\lambda_n(A))_{n\in\D{N}^*}$ is arranged increasingly, i.e., $\lambda_n(A)\leq\lambda_{n+1}(A)$ for any $n\in\D{N}^*$.

\begin{assumptionone*}
The off-diagonal entries $a_l(n)$, $1\leq l\leq m$ satisfy the asymptotics
\begin{equation}\label{17}
a_l(n)=c_ln^{\delta_l}+\ord(n^{\delta-1})\text{ as }n\to\infty,
\end{equation}
where $\delta_l$, $c_l$, $l=1,\dots,m$ are some fixed real numbers and $\delta\geq\max\accol{\delta_1,\dots,\delta_m}$.
\end{assumptionone*}

\begin{assumptiontwo*}
The diagonal entries $d(n)$ satisfy the asymptotics
\begin{equation}     \label{18}
d(n)=c_0n^{\delta_0}+cn^{\delta_0-1}+\ord(n^{\delta_0-2})\text{ as }n\to\infty,
\end{equation}
where $\delta_0>0$, $c_0>0$ and $c\in\D{R}$ are fixed.
\end{assumptiontwo*}

Our main result is the following

\begin{theorem}    \label{thm:11}
Let $A=D+A'$ be defined by \eqref{11}--\eqref{16}. If both assumptions \emph{(H1), (H2)} hold and if $\kappa\coloneqq\delta_0-\delta>0$, then
\begin{equation}  \label{19}
\lambda_n(A)=d(n)+\ord(n^{\delta-\kappa})\text{ as }n\to\infty.
\end{equation}
\end{theorem}

\subsection{Comments} \label{sec:13}
\subsubsection*{a} 
If \eqref{17} is replaced by the weaker condition
\[
a_l(n)=\ord(n^{\delta})
\]
then the min-max principle allows us (see Theorem~\ref{thm:31}) to prove
\begin{equation} \label{110}
\lambda_n(A)=d(n)+\ord(n^{\delta})\text{ as }n\to\infty.
\end{equation}
The main purpose of Theorem~\ref{thm:11} is to show that it is possible to replace the estimate \eqref{110} by the improved estimate \eqref{19}.
\subsubsection*{b} 
For any fixed $j\in\D{Z}$ the assumptions of Theorem \ref{thm:11} imply
\begin{equation}   \label{111}
\frac{a_l(n+j)}{d(n)}=\ord(n^{-\kappa})\text{ as }n\to\infty.
\end{equation}
We observe that the assertion of Theorem~\ref{thm:11} holds for any fixed $\kappa>0$ while all papers \cites{BNS,BZ,JN,M1,M2} assume $\kappa>1$.
\subsubsection*{c} 
We observe that
\begin{equation}   \label{112}
d(n+1)-d(n)\sim\delta_0c_0n^{\delta_0-1}\text{ as }n\to\infty
\end{equation}
and we can treat the case $0<\delta_0<1$ when $d(n+1)-d(n)\to 0$ as $n\to\infty$, while the papers \cites{BNS,JN,M1,M2} assume that
\[
\liminf_{n\to\infty}\bigl(d(n+1)-d(n)\bigr)>0.
\]
\subsubsection*{d} 
Theorem~\ref{thm:11} will be obtained as a special case of more general estimates described in Section~\ref{sec:5} (see Theorems~\ref{thm:51} and \ref{thm:52}). In Section~\ref{sec:54} we give asymptotic estimates of eigenvalues for some cases of not power-like entries.

\subsection{Contents} \label{sec:14}
In Section~\ref{sec:2} we check that the operator $A$ is well defined under assumption~\eqref{14} and its resolvent is compact under assumption \eqref{16}.

In Section~\ref{sec:3} we show how the min-max principle ensures the estimate \eqref{110} if \eqref{17} is replaced by the weaker condition $a_l(n)=\ord(n^{\delta})$.

In Section~\ref{sec:4} we present basic ingredients of our approach based on the construction of operators which are unitarily similar to $A$ with smaller off-diagonal entries. A similar idea is often used to investigate spectral asymptotics of self-adjoint problems defined by a linear PDE, e.g., in relation with the semi-classical approximation in Quantum Mechanics.

In Section~\ref{sec:51} we state a ``general estimate'' (Theorem~\ref{thm:51}). In Section~\ref{sec:52} we derive Theorem~\ref{thm:52} which is an application of the general estimate to power-like entries. In Section~\ref{sec:53} we easily check that Theorem~\ref{thm:11} is a special case of Theorem~\ref{thm:52} and in Section~\ref{sec:54} we apply we apply Theorem~\ref{thm:51} to cases where the entries have different asymptotic behaviors. 

Finally, we complete the proof of the general estimate in Section~\ref{sec:6}. 

The main result (Theorem~\ref{thm:11}) stated above derives from Theorem~\ref{thm:51} as follows:
\[
\left.\begin{array}{c}
\text{Theorem~\ref{thm:51}}\\
\text{Lemma~\ref{lem:42}}
\end{array}\!\!\right\}\implies\text{Theorem~\ref{thm:52}}\implies\text{Theorem~\ref{thm:11}}.
\]

\section{The operator $A$} \label{sec:2}

\begin{proposition} \label{prop:21}
Let $A=A'+D$ be the operator defined by \eqref{11}-\eqref{13} and \eqref{15}.
\begin{enumerate}[\rm(i)]
\item
If \eqref{14} holds, then $A$ is self-adjoint.
\item
If moreover \eqref{16} holds, then $A$ has compact resolvent.
\end{enumerate}
\end{proposition}

It is well known (see \cite{R-S}) that \eqref{15} defines a self-adjoint operator $\C{D} \to l^2$ provided $A'$ has zero relative bound with respect to $D$, i.e., if for any $\varepsilon>0$ there is $C_{\varepsilon}>0$ such that
\begin{equation} \label{21}
\norm{A'x}_{l^2}\leq\varepsilon\norm{Dx}_{l^2}+C_{\varepsilon}\norm{x}_{l^2}\,\text{ for } x\in\C{D}.
\end{equation} 

Before starting the proof of Proposition~\ref{prop:21} we introduce some notations. We recall that the scalar product in $l^2$ is defined by $\scal{x,y}=\sum_{k=1}^{\infty} \overline{x(k)}y(k)$ and we denote by $(\vece_n)_{n=1}^{\infty}$ the canonical basis of $l^2$, i.e., $\vece_n(j)=\delta_{j,n}$ where $\delta_{n,n}=1$ and $\delta_{j,n}=0$ for $j\neq n$.

Then we observe  (see \cite{R-S}) that  it suffices to show \eqref{21} for $x\in c_{00}$, where $c_{00}$ is the linear subspace of $l^2$ generated by the canonical basis, i.e.,
\begin{equation}
x\in c_{00} \Longleftrightarrow\#\{j:x(j)\neq 0\}<\infty.
\end{equation}
We denote by $\C{B}(l^2)$ the algebra of bounded linear operators on $l^2$ with the norm
\begin{equation}
\norm{T}=\sup_{\norm{x}_{l^2}\leq 1}\norm{Tx}_{l^2}.
\end{equation}
The shift operator $S\in\C{B}(l^2)$ is defined by
\begin{equation}
S\vece_n=\vece_{n+1}
\end{equation}
and for any $a\colon\D{N}^*\to\D{C}$ we denote by $a(\Lambda)$ the closed operator in $l^2$ given by
\begin{equation}
\left(a(\Lambda)x\right)(j)=a(j)x(j)\text{ for }x\in c_{00}.
\end{equation}
We can then rewrite the definition of $A'$ in the form
\begin{equation}
A' x=\sum_{1\leq l\leq m}\left(S^la_l(\Lambda)+a_l(\Lambda)S^{l*}\right)x
\end{equation}
where $x\in c_{00}$ and $S^{l*}$ is the adjoint of $S^l$.


\begin{proof}
(i) It suffices to show \eqref{21} for $x\in c_{00}$. For $l=1,\dots,m$ we denote
\begin{equation} \label{27} 
A_l'\coloneqq S^la_l(\Lambda)+a_l(\Lambda)S^{l*}.
\end{equation}
For arbitrary $x\in c_{00}$ we can write
\begin{equation} \label{28}
\norm{A_l'x}_{l^2}^2\leq\left(||S^la_l(\Lambda)x||_{l^2}+||a_l(\Lambda)S^{*l}x||_{l^2}\right)^2.
\end{equation}
Since $S^la_l(\Lambda)S^{*l}=a_l(\Lambda-l)$ with the convention that $a_l(j-l)=0$ if $j\leq l$ and $S^l$ is an isometry, the right-hand side of \eqref{28} can be estimated from above by
\begin{equation}
2||S^la_l(\Lambda)x||_{l^2}^2+2||a_l(\Lambda)S^{*l}x||_{l^2}^2=
2||a_l(\Lambda)x||_{l^2}^2+2||a_l(\Lambda-l)x||_{l^2}^2.
\end{equation}
Then taking $y=(\ii+D)x$ we obtain
\begin{equation}
||A_l'(\ii+D)^{-1}y||_{l^2}^2\leq\sum_j|b_l(j)y(j)|^2
\end{equation}
with
\begin{equation} \label{211}
b_l(j)={\left( 2\,\frac{a_l(j)^2+a_l(j-l)^2}{1+d(j)^2}\right) }^{\! 1/2}\xrightarrow[j\to\infty]{}0.
\end{equation}
For $N\in\D{N}^*$ we denote by ${\Pi}_N$ the orthogonal projection onto $\{ \vece_n\} {}_{1\leq n\leq N}$ and ${\Pi}'_N:=I-{\Pi}_N$. Then \eqref{211} 
implies $||A_l'(\ii+D)^{-1}{\Pi}'_N||\to 0$ as $N\to\infty$. 
Thus for a given $\varepsilon >0$ we can find 
$N(\varepsilon )\in \D{N}$ such that 
\[
||A_l'(\ii+D)^{-1}{\Pi}'_{N(\varepsilon )}y||_{l^2}\leq\varepsilon ||y||_{l^2} 
\]
and we deduce $||A_l' x||_{l^2}\le \varepsilon ||(D+\ii ){\Pi}'_{N(\varepsilon )}x||_{l^2}+||A_l'{\Pi}_{N(\varepsilon )}||\, ||x||_{l^2}$.  

(ii) The operator $A_l'(\ii+D)^{-1}$ is compact as limit of finite rank operators $A_l'(\ii+D)^{-1}{\Pi}_N$ in the norm of $\C{B}(l^2)$. Then $(\ii+A)^{-1}A_l'(\ii+D)^{-1}=(\ii+D)^{-1}-(\ii+A)^{-1}$ is compact and compactness of $(\ii+D)^{-1}$ (due to \eqref{16}) implies that $(\ii+A)^{-1}$ is compact.
\end{proof} 

\section{Asymptotics by min-max principle} \label{sec:3}
\subsection{Statement} \label{sec:31}

In what follows for a sequence $x(n)$ we will use the notation
\begin{equation}   \label{31}
(\Delta x)(n)\coloneqq x(n+1)-x(n).
\end{equation}
The purpose of this section is to prove the following

\begin{theorem}    \label{thm:31}
Let $A=D+A'$ be given by \eqref{11}-\eqref{16}. Assume moreover that there exist $C>0$, $\delta\in\D{R}$, $\kappa>0$ satisfying $\delta+\kappa>0$ and $n_0\in\D{N}$ such that
\begin{align} \label{2120}
&(\Delta d)(n)\geq C^{-1}n^{{\delta}+\kappa-1}\text{ for }n>n_0,
\shortintertext{and, for $l=1,\dots,m$,}
\label{33}
&a_l(n)=\ord(n^{\delta})\text{ as }n\to\infty.
\end{align}
Then one has the large $n$ asymptotic formula
\begin{equation}  \label{2130}
\lambda_n(A)=d(n)+\ord(n^{\delta}).
\end{equation}
\end{theorem}

We observe that due to \eqref{2120} there exists $c>0$ and $n_1$ such that
\begin{equation}   \label{35}
d(n)\geq cn^{\delta+\kappa}\text{ for }n>n_1
\end{equation}
and \eqref{33} with \eqref{35} imply \eqref{111}.

\subsection{Auxiliary estimates} \label{sec:32}

\begin{lemma}   \label{lem:32}
Assume that $\alpha\colon\D{N}^*\to\D{R}$ satisfying the two conditions
\begin{align}  \label{212}
&\alpha(j)\geq\sum_{1\leq l\leq m}(|a_l(j)|+|a_l(j-l)|),\\
\label{213}
&\frac{\alpha(j)}{d(j)}\to 0\text{ as }j\to\infty.
\end{align}
Then for every $n\in\D{N}^*$ the estimate
\begin{align}  \label{214}
&d^-_n \leq\lambda_n(A) \leq d^+_n
\shortintertext{holds with}
\label{eq:215}
&d^-_n\coloneqq\inf_{j\geq n}\accol{d(j)-\alpha(j)},\\
\label{eq:2155}
&d^+_n\coloneqq\sup_{j\leq n}\accol{d(j)+\alpha(j)}.
\end{align}
If moreover there exists $j_0\in\D{N}^*$ such that
\begin{equation} \label{216}
\abs*{(\Delta\alpha)(j)}\leq(\Delta d)(j)\text{ for } j\geq j_0,
\end{equation}
then there exists $n_1\in\D{N}^*$ such that
\begin{equation}  \label{2165}
\abs{\lambda_n(A)-d(n)}\leq\alpha(n)\text{ for }n\geq n_1.
\end{equation}
\end{lemma}

\begin{proof} 
Let $V_n$ denote the linear subspace generated by $\accol{\vece_j}_{1\leq j\leq n}$ and $V_n^{\perp}$ denote its orthogonal complement in $l^2$. Then suitable versions of the min-max principle give
\begin{equation} \label{217}
\inf_{\substack{x\in\C{D}\cap V_{n-1}^{\perp}\\
      \norm{x}_2\leq 1}}\scal{Ax,x}\leq\lambda_n(A)\leq\sup_{\substack{x\in V_n\\
      \norm{x}_2\leq 1}}\scal{x,Ax}.
\end{equation}
Let $A_l'$ be as in \eqref{27}. Then writing
\begin{equation}  \label{218}
\scal{x,A_l'x} =\sum_j  a_l(j)x(j+l){\overline {x(j)}}
 + \sum_k a_l(k-l) x(k-l){\overline {x(k)}}
\end{equation}
with $k=j+l$ we can estimate $|\scal{x,A_l'x} |$  by
\begin{equation}  \label{219}
  \sum_j 2|a_l(j)|\, |x(j+l)x(j)|\leq  \sum_j |a_l(j)|(|x(j+l)|^2+|x(j)|^2).
\end{equation}
Therefore the right-hand side of \eqref{219} can be written in the form
\begin{equation}  \label{220}
\sum_k |a_l(k)||x(k+l)|^2+\sum_j |a_l(j)||x(j)|^2 =\sum_j (|a_l(j-l)|+|a_l(j)|)|x(j)|^2
\end{equation}
and we obtain
\begin{equation} \label{221}
|\scal{x,A'x} |\leq\sum_j  \alpha(j) |x(j)|^2.
\end{equation}
 Next we observe that \eqref{221}  implies
\begin{equation} \label{222}
 \scal{x,\, (d(\Lambda)-\alpha(\Lambda))x} \leq\scal{x,\, Ax} \le
 \scal{x,\, (d(\Lambda)+\alpha(\Lambda))x}
\end{equation}
and using \eqref{222} we can estimate the right-hand side of \eqref{217} by
\[
\sup_{\substack{x\in V_n\\\norm{x}_2\leq 1}}\scal{x,(d(\Lambda)+\alpha(\Lambda))}=d_n^+.
\]
To complete the proof of \eqref{214} note that
\[
\inf_{\substack{x\in \C{D}\cap V_{n-1}^{\perp}\\\norm{x}_2\leq 1}}\scal{x,(d(\Lambda)-\alpha(\Lambda))x}=d_n^-
\]
is smaller than the left-hand side of \eqref{217}. In order to show \eqref{2165} we observe that \eqref{216} implies
\begin{equation} \label{226}
d(j)-\alpha(j) \leq d(j+1)-\alpha(j+1)   \text{ for } j\geq n_0,
\end{equation}
\begin{equation} \label{227}
d(j)+\alpha(j)\leq d(j+1)+\alpha(j+1)  \text{ for } j\geq n_0,
\end{equation}
and consequently $d^{\pm}_n=d(n)\pm \alpha(n)$ for $n\geq n_1$, hence \eqref{2165} follows from \eqref{214}.
\end{proof}

\subsection{Proof of Theorem~\ref{thm:31}} \label{sec:33}

\begin{proof} 
Due to \eqref{110} there exists $C_0>0$ such that \eqref{212}, \eqref{213} hold with
\begin{equation} \label{224}
\alpha(j) \coloneqq C_0 j^{\delta}
\end{equation}
and \eqref{19} ensures the estimate
\begin{equation} \label{225}
\left|(\Delta\alpha)(j)\right|\sim |\delta| C_0j^{\delta -1}\leq\left|\delta\right|C_0 Cj^{-\kappa} (\Delta d)(j)\text{ for } j\geq j_0.
\end{equation}
Since $\left|\delta\right|C_0Cj^{-\kappa}\to 0$ as $j\to\infty$, it is clear that \eqref{225} implies \eqref{216} if $j_0$ is large enough. Thus \eqref{2165} holds with $\alpha(j)$ given by \eqref{224} and the proof of \eqref{112} is complete.
\end{proof} 

\section{Basic ingredients of the approach} \label{sec:4}
\subsection{Main ideas} \label{sec:41}

We write the following formal development of the conjugate 
\begin{equation} \label{301}
B_n\coloneqq\eul^{-\ii P_n}A\eul^{\ii P_n}=A+[A,\ii P_n]+\frac{1}{2}[[A,\ii P_n],\ii P_n]+\dots 
\end{equation} 
where $P_n$ is self-adjoint and of finite rank for simplicity. Then $\lambda_n(A)=\lambda_n(B_n)$ and we want to determine $P_n$ so that $B_n$ is close to a diagonal operator at least for the entries with indices ranging between $n-\tau_n$ and $n+\tau_n$ where  
$(\tau_n)_{n=1}^{\infty}$ is a sequence of positive integers 
satisfying 
 \begin{align}   \label{3020}  
&\tau_n\leq\tau_{n+1}\text{ for }n\in\D{N}^*, 
 \\ \label{3021} &\tau_n \xrightarrow[n\to\infty]{} \infty, \\ 
\label{3022} &n-2\tau_n\xrightarrow[n\to\infty]{}  \infty.
\end{align}  
We remark that in the proof of Theorem~\ref{thm:11} we take  $\tau_n=\left\lfloor\frac{1}{4}\,n\right\rfloor$, where $\lfloor s\rfloor\coloneqq\max\{k\in\D{Z}:k\leq s\}$ means the integer part of $s$.

Further on $\chi\in C^1(\D{R})$ is a fixed function satisfying 
\begin{align*} 
&0\leq\chi\leq 1,\\
&\chi(s)=1\text{ for }s\in\croch{-1,1},\\
&\chi(s)=0\text{ for }s\notin\croch{-2,2}.
\end{align*}  
Then we write the decomposition
\begin{align} \label{3035}
&a_l(j)=a_{n,l}(j)+\tilde{a}_{n,l}(j)
\shortintertext{with}
\label{302}
&a_{n,l}(j)\coloneqq a_l(j)\,\chi\!\left(\frac{j-n}{\tau_n}\right),\\
\label{303}
&\tilde{a}_{n,l}(j)\coloneqq a_l(j)\,(1-\chi)\!\left(\frac{j-n}{\tau_n}\right)
\end{align} 
and the corresponding decomposition 
\begin{align} \label{304}
&A'=A_n+\tilde A_n,
\shortintertext{where}
\label{305}
&A_n=\sum_{1\leq l\leq m}\left(S^la_{n,l}(\Lambda)+a_{n,l}(\Lambda)S^{l*}\right),\\
\label{306}
&\tilde A_n=\sum_{1\leq l\leq m}\left(S^l\tilde{a}_{n,l}(\Lambda)+\tilde{a}_{n,l}(\Lambda)S^{l*}\right).  
\end{align} 
Using \eqref{304} we rewrite \eqref{301} in the form
\begin{equation} \label{307}
B_n=\eul^{-\ii P_n}A\eul^{\ii P_n}=D+\tilde A_n+A_n+[D,\ii P_n]+W_n,    
\end{equation} 
where $W_n$ is considered as a lower order error. However due to    
\begin{equation} \label{3075}
n-2 \tau_n\le j \le n+2\tau_n\implies{\tilde a}_{n,l}(j)=0
\end{equation} 
it is easy to see that for $n$ large enough we have $(D+{\tilde A}_n)\eul_n=D\eul_n=d(n)\eul_n$, i.e., $d(n)$ is an eigenvalue of $D+{\tilde A}_n$. Then in Section~\ref{sec:43}, Lemma~\ref{lem:42} we show that $d(n)$ is the $n$-th eigenvalue of $D+\tilde A_n$ provided $n$ is large enough and the entries are sufficiently regular. Next we choose $P_n$ satisfying the commutator equation 
\begin{equation} \label{309} 
A_n+\croch{D,\ii P_n}=0,
\end{equation} 
hence the expression \eqref{307} takes the form 
\[
B_n=\eul^{-\ii P_n}A\eul^{\ii P_n}=D+\tilde A_n+W_n     
\]
and using $\lambda_n(A)=\lambda_n(B_n)$, $d(n)=\lambda_n(D+\tilde A_n)$ with the min-max principle we obtain   
\begin{equation} \label{3011} 
\abs{\lambda_n(A)-d(n)}=\abs{\lambda_n(B_n)-\lambda_n(D+\tilde A_n)}\leq\norm{B_n-(D+\tilde A_n)} 
\end{equation} 
for $n>{\tilde n}_0$. 

\begin{lemma}   \label{lem:41}
Let $P_n$ be a finite rank self-adjoint operator satisfying $A_n=\ii\croch{P_n,D}$. If $A'=A_n+\tilde A_n$ and $B_n=\eul^{-\ii P_n}A\eul^{\ii P_n}$, then
\begin{equation} \label{321}
\norm{B_n-(D+\tilde A_n)}\leq\norm{\croch{P_n,\tilde A_n}}+\frac{1}{2}\norm{\croch{P_n,A_n}}.
\end{equation}
\end{lemma}

All our results will follow from suitable estimates of the right-hand side of \eqref{321}, i.e., estimates of norms of commutators. A general estimate is stated in Section~\ref{sec:5} and the norms of commutators from the right-hand side of \eqref{321} are estimated in Section~\ref{sec:6}.

\begin{proof}
We introduce
\begin{equation} \label{324}
\tilde{B}_n\coloneqq\eul^{-\ii P_n}\tilde A_n\eul^{\ii P_n}-\tilde A_n
\end{equation}
and we observe that
\[
\tilde{B}_n=\int_0^1\frac{\dd}{\dd s}\left(\eul^{-\ii sP_n}\tilde A_n\eul^{\ii sP_n}\right)\dd s=\int_0^1\eul^{-\ii sP_n}\ii[\tilde A_n, P_n]\eul^{\ii sP_n}\dd s.
\]
Since for $s\in\D{R}$ the operators $\eul^{\ii sP_n}$ are unitary, $\norm{\eul^{\ii sP_n}}=1$ and we find
\begin{equation} \label{327}
\norm{\tilde{B}_n}\leq \norm{[\tilde A_n, P_n]}.
\end{equation}
Next for $s\in\D{R}$ we introduce
\begin{equation} \label{325}
G_n(s)\coloneqq\eul^{-\ii sP_n}(D+\ii[sP_n,D])\eul^{\ii sP_n}-D 
\end{equation} 
and we observe that
\begin{align*}
\frac{\dd}{\dd s}G_n(s)
&=\eul^{-\ii sP_n}\bigl(\ii\croch*{D+\ii\croch{sP_n,D},P_n}+\ii\croch{P_n,D}\bigr)\eul^{\ii sP_n}\\
&=\eul^{-\ii sP_n}s\croch*{\croch{D,P_n},P_n}\eul^{\ii sP_n}.
\end{align*}
Therefore
\[
G_n(1)=\int_0^1\eul^{-\ii sP_n}s\croch*{\croch{D,P_n},P_n}\eul^{\ii sP_n}\dd s
\]
and we can estimate
\begin{equation} \label{328}
\norm{G_n(1)}\leq\int_0^1 s\,\norm{\croch*{\croch{D,P_n},P_n}}\dd s=\frac{1}{2}\norm{\croch*{\croch{D,P_n},P_n}}.
\end{equation}
However 
\[
B_n =\eul^{-\ii P_n}(D+A_n)\eul^{\ii P_n}+\eul^{-\ii P_n}\tilde A_n\eul^{\ii P_n}=(G_n(1)+D)+(\tilde{B}_n+\tilde A_n),
\] 
hence   
\begin{equation} \label{329}
||B_n-(D+\tilde A_n)||=||\tilde{B}_n+G_n(1)||\leq ||\tilde{B}_n||+||G_n(1)||.
\end{equation} 
To complete the proof it remains to estimate the right-hand side of  \eqref{329} 
using \eqref{327}-\eqref{328}. 
\end{proof}

\subsection{Equality $\BS{d(n)=\lambda_n(D+\tilde A_n)}$} \label{sec:43}

In this section we give sufficient conditions to ensure the equality $d(n)=\lambda_n(D+\tilde A_n)$ used in estimate \eqref{3011}. 

We consider a sequence of positive integers $(\tau_n)_{n=1}^{\infty}$ satisfying \eqref{3020}-\eqref{3022} and $n_0>0$. 
We assume that the inequalities 
\begin{align} \label{40}
&d(n)<d(n+1), \\
\label{41}
&4m\max \{ |a_l(n)|,\, |a_l(n+m-\tau_n)| \}  
   \le d(n)-d(n+m-\tau_n ) 
\end{align} 
hold for $n\ge n_0$ and $l=1,\dots ,m$. 
\begin{lemma}   \label{lem:42} 
Assume that \eqref{40}, \eqref{41} hold for $n\ge n_0$. If ${\tilde A}_n$ is defined by means of ${\tilde a}_{n,l}$ and $\chi$ as in Section~\ref{sec:41}, then there is ${\tilde n}_0\in \D{N}$ such that 
\begin{equation}     \label{308}
d(n)=\lambda_n(D+\tilde A_n)\text{ for }n> {\tilde n}_0.
\end{equation}
\end{lemma}

\begin{proof}
We introduce
\begin{equation} \label{414}
\tilde\alpha_n(j):=\sum_{1\leq l\leq m}(|\tilde a_{n,l}(j)|+|\tilde a_{n,l}(j-l)|)
\end{equation}
and observe that Lemma~\ref{lem:32} allows us to estimate
\begin{align} \label{415}
&{\tilde d}_n^-\leq\lambda_n(D+{\tilde A}_n)\leq{\tilde d}^+_n
\shortintertext{with}
\label{416}
&{\tilde d}_n^-=\inf_{j\geq n}\accol{d(j)-\tilde\alpha(j)},\\
\label{417}
&{\tilde d}_n^+=\sup_{j\leq n}\accol{d(j)+\tilde\alpha(j)}.
\end{align} 

\emph{First step}. We claim that 
\begin{equation} \label{4181}
j\ge n+\tau_n\implies{\tilde a}_n(j)\le d(j)-d(n) 
\end{equation} 
holds for $n\ge n_0+m$. Indeed, replacing $n$ by $j-i$ in \eqref{41} we obtain  
\begin{equation} \label{4191}
j-i\ge n_0\implies  4m|a_l(j-i)|\le d(j-i)-d(j-i+m-\tau_{j-i}), 
\end{equation}
hence for $0\le i\le m$, $j\ge n+\tau_n$ we have $j-i+m-\tau_{j-i}\ge j-\tau_n\ge n$ and applying \eqref{40} we find 
\begin{equation} \label{4201}
j\ge n+\tau_n\implies d(j-i)-d(j-i+m-\tau_{j-i})\le d(j)-d(n) 
\end{equation}
for $0\le i\le m$, $n\ge n_0+m$. Combining $|{\tilde a}_{n,l}(j)|\le |a_l(j)|$  with \eqref{4191} and \eqref{4201} we obtain \eqref{4181}.

\emph{Second step}. We claim that 
\begin{equation} \label{4182}
 n_0+m\le j\le n-\tau_n\implies{\tilde a}_n(j)\le d(n)-d(j).
\end{equation} 
Indeed, replacing $n$ by $j-i+\tau_n$ in \eqref{41} with $j\ge n_0+i$ we find
\begin{equation} \label{4192}
4m|a_l(j-i)|\le d(j-i+\tau_{j-i+\tau_n})-d(j-i+m), 
\end{equation}
hence for $0\le i\le m$, $j\le n-\tau_n$ we have $j-i+\tau_{j-i+\tau_n}\le j+\tau_n\le n$ and applying \eqref{40}  we find  
\begin{equation} \label{4202}
 n_0+m\le j\le n-\tau_n\implies  
d(j-i+\tau_{j-i+\tau_n})-d(j-i+m)\le d(n)-d(j) 
\end{equation}
for $0\le i\le m$. As before, \eqref{4182} follows from \eqref{4192} and \eqref{4202}. 

Next we observe that by definition $n-\tau_n\le j\le n+\tau_n\implies{\tilde a}_n(j)=0$, hence 
\begin{align} \label{450}
n_0+m\le j\le n&\implies {\tilde a}_n(j)\le d(n)-d(j)\implies d(j)+{\tilde\alpha}_n(j)\le d(n),\\   
\label{451}
n_0+m\le n\le j&\implies{\tilde a}_n(j)\le d(j)-d(n)\implies d(n)\le d(j)-{\tilde\alpha}_n(j). 
\end{align}
Therefore one can choose ${\tilde n}_0$ large enough to ensure ${\tilde d}^-_n=d(n)={\tilde d}^+_n$ for $n\ge {\tilde n}_0$ and \eqref{308} follows from \eqref{415}.
\end{proof}

\section{A general estimate} \label{sec:5}
\subsection{Statement} \label{sec:51}

We fix a sequence of positive integers $(\tau_n)_{n=1}^{\infty}$ satisfying \eqref{3020}-\eqref{3022} and for $s\ge 0$ we denote 
\begin{align}  \label{44}
\alpha_s(j)&\coloneqq\max_{\substack{1\leq l\leq m\\|i|\leq s}}|a_l(j+i)|,\\
\label{46}
\tilde\alpha_{n,s}(j)&\coloneqq\frac{2\alpha_s(j)}{\tau_n}+
\max_{\substack{1\leq l\leq m\\|i|\leq s}}|(\Delta a_l)(j+i)|,\\
\label{43}
\gamma_s(j)&\coloneqq\min_{|i|\leq s+1}(\Delta d)(j+i),\\
\label{45}
\tilde\gamma_s(j)&\coloneqq\max_{|i|\leq s}|({\Delta}^2 d)(j+i)|,
\end{align} 
where $(\Delta^2d)(n)=(\Delta d)(n+1)-(\Delta d)(n)=d(n+2)-2d(n+1)+d(n)$.

\begin{theorem}[general estimate]   \label{thm:51}
Let $A$ be defined by \eqref{11}-\eqref{16}. Let $(\tau_n)_{n=1}^{\infty}$,    $\alpha_s$, $\gamma_s$, $\tilde\alpha_{n,s}$, $\tilde\gamma_s$ be as above. If \eqref{308} holds and
\begin{equation} \label{47}
\rho_n(j)\coloneqq 2\tilde\alpha_{n,4m}(j)
\frac{\alpha_{4m}(j)}{\gamma_{4m}(j)}+
m\tilde\gamma_{5m}(j)\frac{\alpha_{4m}(j)^2}{\gamma_{4m}(j)^2},
\end{equation}
then there is $n_1$ such that for $n\geq n_1$ one has the estimate
\begin{equation} \label{48}
|\lambda_n(A)-d(n)|\leq 15 m^3\sup_{|i|\le 2\tau_n +4m}{\rho}_n(n+i).
\end{equation}
\end{theorem} 

\begin{proof}
This general estimate will be proved in Section~\ref{sec:6}.
\end{proof}

\subsection{Application}   \label{sec:52}

We check that Theorem~\ref{thm:51} implies

\begin{theorem}   \label{thm:52}
Let $A$ be defined by \eqref{11}-\eqref{16}. Assume that there exist $C>0$, $\delta\in\D{R}$, $\kappa>0$ satisfying $\delta+\kappa>0$ and $n_0\in\D{N}$ such that
\begin{alignat}{2} \label{520}
C^{-1}n^{\delta+\kappa-1}\le (\Delta d)(n)&\leq Cn^{\delta+\kappa-1}&\;\;&\text{for }n\geq n_0,\\
\label{521}
(\Delta^2d)(n)&=\ord(n^{\delta+\kappa-2})&&\text{as }n\to\infty,
\\
\label{5220}
a_l(n)&=\ord(n^{\delta})&&\text{as }n\to\infty,\ l=1,\dots,m.\\
\label{522}
(\Delta a_l)(n)&=\ord(n^{\delta-1})&&\text{as }n\to\infty,\ l=1,\dots,m.
\end{alignat}
Then one has the estimate
\begin{equation} \label{480}
\lambda_n(A)=d(n)+\ord(n^{\delta-\kappa})\text{ as }n\to\infty.
\end{equation}
\end{theorem}  

\begin{proof} 
Due to \eqref{520} and \eqref{521} there exist $C_0>c_0>0$ and $n_1\in\D{N}$ satisfying  
\begin{equation}  
c_0n^{\delta+\kappa}\leq d(n)\leq C_0n^{\delta+\kappa}\text{ for }n\geq n_1.
\end{equation} 
In order to ensure \eqref{308} we will check that the assumptions of Lemma~\ref{lem:42}  hold if $\tau_n=n+m-\lfloor n{\varepsilon}_0\rfloor$ where ${\varepsilon}_0>0$ is fixed sufficiently small. For this purpose we introduce $c_1:=c_0-C_0({\varepsilon}_0)^{\delta+\kappa}$ and estimate  
\begin{equation} \label{1170}
d(n)-d(\lfloor n{\varepsilon}_0\rfloor)\ge c_0n^{\delta +\kappa}-C_0(n{\varepsilon}_0)^{\delta+\kappa}=c_1n^{\delta +\kappa}\text{ for } n>n_1.
\end{equation} 
Let ${\varepsilon}_0>0$ be small enough to ensure $c_1>0$. Then it is clear that   \eqref{41} follows from \eqref{1170} and \eqref{5220}. 

Thus all assumptions of Lemma \ref{lem:42} hold and it remains to apply Theorem \ref{thm:51}. Using \eqref{520}-\eqref{522} we can find a constant $C_1$ such that
\begin{align}  \label{441}
\inf_{|i|\le 2\tau_n +4m} \gamma_{4m}(n+i)&\geq C_1^{-1}n^{\delta+\kappa-1},\\
\label{442}
\sup_{|i|\le 2\tau_n +4m} \tilde\gamma_{5m}(n+i) &\leq C_1n^{\delta+\kappa-2},\\
\label{443}
\sup_{|i|\le 2\tau_n +4m} \alpha_{4m}(n+i)&\leq C_1n^{\delta},\\
\label{444}
\sup_{|i|\le 2\tau_n +4m}\tilde\alpha_{n,4m}(n+i)&\leq C_1n^{\delta-1}
\end{align}
with $\gamma_s$, $\tilde\gamma_s$, $\alpha_s$, $\tilde\alpha_{n,s}$, given  
by \eqref{44}-\eqref{45}. Therefore,
\begin{align}  \label{445}
\tilde\alpha_{n,4m}(j)\frac{\alpha_{4m}(j)}{\gamma_{4m}(j)}&\leq C_1^3n^{\delta-1}\frac{n^{\delta}}{n^{\delta+\kappa-1}}=
C_1^3n^{\delta-\kappa},\\
\label{446}
\tilde\gamma_{5m}(j)\frac{\alpha_{4m}(j)^2}{\gamma_{4m}(j)^2}&\leq C_1^5n^{\delta+\kappa-2}\frac{n^{2\delta}}{n^{2(\delta+\kappa-1)}}=
C_1^5n^{\delta-\kappa}
\end{align}
hold when $|j-n|\le 2\tau_n +4m$ and it is clear that
\begin{equation}  \label{447}
\sup_{|i|\le 2\tau_n +4m}\rho_n(n+i)=\ord(n^{\delta-\kappa})
\end{equation}
if $\rho_n(j)$ is given by \eqref{47}. We conclude that \eqref{480} follows from \eqref{48} and \eqref{447}.
\end{proof}

\subsection{Proof of Theorem~\ref{thm:11}} \label{sec:53}

\begin{proof}
Let $\delta_0=\delta+\kappa$. Then the assumptions (H1) and (H2) imply \eqref{520}-\eqref{522}. Consequently Theorem~\ref{thm:11} follows from Theorem~\ref{thm:52}.
\end{proof}  

\subsection{Other applications of the general estimate} \label{sec:54} 

In this section we consider $d(n)\sim\omega (n)$ where the function $\omega\colon(0,\infty)\to(0,\infty)$ is one of a special type of functions described below.  

\subsubsection{} \label{sec:541} 
We fix $\kappa >0$, $\kappa'\in \D{R}$ and assume 
\begin{align} \label{460}
|a_l(n)|+n|\Delta a_l(n)|&=\ord(n^{-\kappa} (\ln n)^{-\kappa '}\omega (n)),\\ 
\label{461}
d(n)&=\omega (n)\left(1+\ord(n^{-2})\right)  
\end{align} 
 where  
\begin{equation} \label{462}
\omega (\lambda)= c_0 {\lambda}^{{\delta}_0} (\ln \lambda )^{{\delta}'_0}
\end{equation}  
holds with some $c_0>0$, ${\delta}_0>0$, ${\delta}'_0 \in \D{R}$. We observe that the derivatives satisfy ${\omega}^{(k)} (\lambda )\sim c_k {\lambda}^{-k}\omega(\lambda)$ as $\lambda\to\infty$. Since $\Delta \omega (n)=\omega'(n+r_n)$ holds with some $r_n\in[0,\,1]$ and $\omega '(n+r_n)\sim \omega '(n)$ as $n \to\infty$ we easily deduce
\begin{alignat}{2} \label{463}
{\gamma}_s(n)&\sim {\delta}_0n^{-1}\omega (n)&\quad&(n\to\infty),\\
\label{464}
{\gamma}'_s(n)&=\ord(n^{-2}\omega (n))&&(n\to\infty).
\end{alignat} 
Using $\tau_n=\lfloor n/4\rfloor$ we find that Theorem~\ref{thm:51} gives the estimate 
\begin{equation} \label{465}
{\lambda}_n(A)=d(n)\left({1+\ord(n^{-2\kappa} (\ln n)^{-2\kappa '})}\right).  
\end{equation} 
It is easy to see that \eqref{465} still holds when $\kappa =0$ and $\kappa ' >0$.
 
\subsubsection{} \label{sec:541} 
We assume that \eqref{460} holds with some $\kappa>0$, $\kappa'\in\D{R}$, 
\begin{equation}  \label{461a}
d(n)=\omega (n)\,\left({1+\ord(n^{-2} (\ln\lambda )^{-1})}\right) 
\end{equation} 
where 
\begin{equation} \label{462a}
\omega (\lambda )=c_0 (\ln\lambda )^{{\delta}'_0} 
\end{equation}   
holds with some $c_0>0$, ${\delta}'_0>0$. Then computing the derivatives of $\omega$ we find 
\begin{alignat}{2} \label{463a}
{\gamma}_s(n)&\sim {\delta}_0n^{-1}(\ln n)^{-1}\omega (n)&\quad&(n\to\infty),\\
\label{464a}
{\gamma}'_s(n)&=\ord(n^{-2}(\ln n)^{-1}\omega (n))&&(n\to\infty);
\end{alignat}
and using $\tau_n=\lfloor n/4\rfloor$ in Theorem~\ref{thm:51} we obtain the estimate 
\begin{equation} \label{45b}
{\lambda}_n(A)=d(n)\left( {1+\ord(n^{-2\kappa}(\ln n)^{1-\kappa'})} \right). 
\end{equation}  
It is easy to see that \eqref{45b} still holds when $\kappa=0$ and $\kappa'>1$.
 
\subsubsection{} \label{sec:541} 
We assume that $\kappa >0$, $0<\theta< 1$ and  
\begin{align} \label{460c}
|a_l(n)|  +  n^{1-\theta}|\Delta a_l(n)|&=\ord(n^{-\kappa}\omega(n)),\\ 
\label{461c}
d(n)&=\omega(n)\left({1+\ord(n^{\theta -2})} \right)
\end{align} 
where  
\begin{equation} \label{462c} 
\omega(\lambda)=c_0{\lambda}^{{\delta}_0}\,\eul^{c{\lambda}^{\theta}}
\end{equation}  
holds for some $c_0>0$, $c>0$, ${\delta}_0\in \D{R}$. The derivatives satisfy ${\omega}^{(k)}(\lambda )\sim c_k {\lambda}^{-k(1-\theta)}\omega (\lambda )$ as $\lambda\to\infty$ and we deduce 
\begin{alignat}{2} \label{463c}
{\gamma}_s(n)&\sim c\,\theta\,n^{-(1-\theta )}\omega (n)&\quad&(n\to\infty),\\
\label{464c}
{\gamma}'_s(n)&=\ord(n^{-2(1-\theta )}\omega (n))&&(n\to\infty).
\end{alignat}
Let $\tau_n=\lfloor n^{1-\theta}{\varepsilon}_0\rfloor$ with ${\varepsilon}_0>0$ small 
enough. Then $d(n)-d(n+m-\tau_n)\ge d(n)/2$ for $n\ge n_1$ and Theorem\ref{thm:51} ensures 
\begin{equation}  
{\lambda}_n(A)=d(n)\left({1+\ord(n^{-2\kappa})}\right).
\end{equation} 
\section{Proof of Theorem \ref{thm:51}} \label{sec:6}
\subsection{Estimates of commutators} \label{sec:61}

For $n\in\D{N}^*$, $l\in\D{Z}$ we consider $p_{n,l}$, $a_{n,l}\colon\D{Z}\to\D{R}$ satisfying
\begin{align} \label{331}
p_{n,-l}(j)&=p_{n,l}(j)\text{ and }a_{n,-l}(j)=a_{n,l}(j),\\
\label{332}
p_{n,l}(j)&=a_{n,l}(j)=0 \text{ when } j\leq 0,\\
\label{333}
p_{n,l}(j)&=a_{n,l}(j)=0 \text{ when } |l|>m,
\end{align}
where $m\in\D{N}^*$ is fixed. We assume moreover
\begin{equation} \label{334}
p_{n,l}(j)=0 \text{ when } |j-n|\ge 2\tau_n,
\end{equation}
where $(\tau_n)_{n=1}^{\infty}$ is as before and consider finite rank self-adjoint operators  
\begin{align}\label{335}
P_n&\coloneqq\sum_{1\leq l\leq m}\left(S^lp_{n,l}(\Lambda)+p_{n,l}(\Lambda)S^{l*} \right),\\
\label{336}
R_n&\coloneqq\ii[P_n,A_n]=\ii(P_nA_n-A_nP_n),  
\end{align} 
where $A_n$ is given by \eqref{305}. For $s\geq 0$, $j\in\D{Z}$ define
\begin{align} \label{337}
\alpha_{n,s}(j)&\coloneqq\max_{\substack{|l|\leq m\\|i|\leq s}}|a_{n,l}(j+i)|,\\
\label{338}
\beta_{n,s}(j)&\coloneqq\max_{\substack{|l|\leq m\\|i|\leq s}}|p_{n,l}(j+i)|,\\
\label{339}
\alpha_{n,s}'(j)&\coloneqq\max_{\substack{|l|\leq m\\|i|\leq s}}|(\Delta a_{n,l})(j+i)|,\\
\label{340}
\beta_{n,s}'(j)&\coloneqq\max_{\substack{|l|\leq m\\|i|\leq s}}|(\Delta p_{n,l})(j+i)|.
\end{align}

\begin{lemma}  \label{lem:61}
Let $A_n$, $P_n$, $R_n$, $\alpha_{n,s}$, $\beta_{n,s}$,  $\alpha_{n,s}'$,
$\beta'_{n,s}$ be as above and
\begin{equation} \label{341}
 \rho_{n,s}(j)\coloneqq \alpha_{n,s}(j)\beta'_{n,s}(j)+
  \alpha_{n,s}'(j)\beta_{n,s}(j).
\end{equation}
Then one has
\begin{equation} \label{342}
\norm{R_n}\leq 10m^3\sup_{|i|\le 2\tau_n +4m} \rho_{n,4m}(n+i).
\end{equation}
\end{lemma}

\begin{proof} 
We can express $R_n$ as
\begin{equation} \label{343}
R_n=r_{n,0}(\Lambda)+\sum_{1\leq k\leq 2m}\left(S^kr_{n,k}(\Lambda)+r_{n,k}(\Lambda)S^{k*}\right), 
\end{equation}
where for $k\geq 0$ we have
\begin{equation} \label{3435}
r_{n,k}(j)=\scal{\vece_{j+k},R_n\vece_j}=\ii\scal{\vece_{j+k},P_nA_n\vece_{j}}-\ii\scal{\vece_{j+k},A_nP_n\vece_j}.
\end{equation} 
For $s\in\D{R}$ we write $s_+=\max\{s,0\}$ and $s_-=(-s)_+$. Then 
\begin{align} \label{344}
\scal{\vece_i,P_n\vece_j}&=p_{n,i-j}(i-(i-j)_+),\\
\label{345}
\scal{\vece_i,A_n\vece_j}&=a_{n,i-j}(i-(i-j)_+),
\end{align}  
and using \eqref{344}, \eqref{345} in 
\begin{align*}
\scal{\vece_{j+k},P_nA_n\vece_j}
&=\sum_l\scal{\vece_{j+k},P_n\vece_{j+l}}\,\scal{\vece_{j+l},A_n\vece_j},\\
\scal{\vece_{j+k},A_nP_n\vece_j}
&=\sum_l\scal{\vece_{j+k},A_n\vece_{j+k-l}}\scal{\vece_{j+k-l},P_n\vece_j},
\end{align*}
we find 
\begin{align} \label{348}
r_{n,k}(j)&=\ii \sum_{1\leq |l|\leq m} r_{n,k,l}(j)
\shortintertext{with}
r_{n,k,l}(j)&=p_{n,k-l}(j+k-(k-l)_+ )a_{n,l}(j-l_-)\notag\\
&\quad-a_{n,l}(j+k-l_+ )p_{n,k-l}(j-(k-l)_-).\notag 
\end{align} 
Moreover $p_{n,k-l}\neq 0 \implies |k-l|\leq m$ and we claim that 
\begin{equation} \label{349}
r_{n,k,l}(j)\ne 0 \implies |j-n|\le 2\tau_n+2m .
\end{equation}  
Indeed, it suffices to use $0\leq k\leq 2m$ and $|k-l|\leq m$ in 
\begin{align*}
p_{n,k-l}(j-(k-l)_-)\ne 0&\implies n-2\tau_n\le 
   j-(k-l)_- \le n+2\tau_n,\\
p_{n,k-l}(j+k-(k-l)_+)\ne 0&\implies n-2\tau_n 
  \le j+k-(k-l)_+\le n+2\tau_n .
\end{align*} 
Then reasoning as in Section~\ref{sec:2} we can estimate  
\begin{equation} \label{350}
\norm{R_n}\leq ||r_{n,0}(\Lambda)||+\sum_{1\leq k\leq 2m}\left(2||r_{n,k}
(\Lambda)||^2+2||r_{n,k}(\Lambda-k)||^2\right)^{\!1/2}  
\end{equation}
and the right-hand side of \eqref{350} can be estimated by
\begin{equation} \label{351}
(4m+1)\,\sup_{j\geq 1}\,\max_{0\leq i\leq k\leq 2m}|r_{n,k}(j-i)|.
\end{equation}
However for $i,i'\in\D{Z}$ such that $i'<i$ we have the expression  
\begin{equation} \label{352}
a_{n,l}(j+i)-a_{n,l}(j+i')=\sum_{i'\leq j'\leq i-1}(\Delta a_{n,l})(j+j'),    
\end{equation}
and using $|k-l_++l_-|=|k-l|\leq m$ we obtain the estimate 
\begin{equation} \label{353}
|a_{n,l}(j+k-l_+ )-a_{n,l}(j-l_-)|\leq m \alpha_{n,s}'(j)
\end{equation} 
with $s=\max\{|l|,|k-l_+|\}\leq\max \{ m,\, k\}\leq 2m$. Thus \eqref{353} holds with $s=2m$ and
\begin{equation} \label{354}
|p_{n,k-l}(j+k-(k-l)_+ )-p_{n,k-l}(j-(k-l)_- )|\leq m\,\beta_{n,2m}'(j) 
\end{equation}
follows similarly. Since $r_{n,k}(j)=r_{n,k}'(j)+r_{n,k}''(j)$ holds with 
\begin{align}  \label{346}
r_{n,k}'(j)&=(p_{n,k-l}(j+k-(k-l)_+ )-p_{n,k-l}(j-(k-l)_-))a_{n,l}(j-l_-),\\ 
\label{347}
r_{n,k}''(j)&=p_{n,k-l}(j-(k-l)_-)(a_{n,l}(j-l_-)-a_{n,l}(j+k-l_+)),
\end{align}
we obtain $|r_{n,k}(j)|\leq m\,\rho_{n,2m}(j)$ and $\norm{R_n}$ can be estimated by
\begin{equation}
2m^2(4m+1)\,\sup_{|j-n|\le 2\tau_n +4m}\, 
\max_{0\leq i\leq 2m}|\rho_{n,2m}(j-i)|.
\end{equation}
To complete the proof of \eqref{342} it remains to use $\rho_{n,2m}(j-i)\leq\rho_{n,2m+|i|}(j)$.
\end{proof}

\subsection{Proof of Theorem \ref{thm:51}} \label{sec:62}

\begin{proof} 
Let $n_0$ be as in \eqref{40} and  for $j\geq 1-l\geq n_0$ denote
\begin{equation} \label{422}
d'_l(j)\coloneqq d(j+l)-d(j).
\end{equation}
Consider $A_n$, $P_n$ given by \eqref{305}, \eqref{306} with $a_{n,l}$ as in 
\eqref{302} and  
\begin{equation} \label{423}
 p_{n,l}(j)=\scal{\vece_{j+l},\, P_n\vece_j} = 
 \ii\,\frac{a_{n,l}(j)}{d'_l(j)}\text{ for } l=1,\dots ,m.
\end{equation}
Then  $R_n\coloneqq[\ii P_n,\, D]$ coincides with $A_n$ due to
\begin{align*}
\ii\scal{\vece_{j+l},\, R_n\vece_{j}} 
 &=\ii\scal{\vece_{j+l},\,P_nD\vece_{j}}-\ii\scal{D\vece_{j+l},\, P_n\vece_j}\\
 &=\ii(d(j)-d(j+l))\scal{\vece_{j+l},\, P_n\vece_j}\\
 &=a_{n,l}(j)=\scal{\vece_{j+l},A_n\vece_{j}}\text{ for } l\geq 0.
\end{align*} 
Thus due to \eqref{3011} and \eqref{321}
\begin{equation} \label{424}
|\lambda_n(A)-d(n)|\leq ||[\tilde A_n,P_n]||+\frac{1}{2}\,||[A_n,P_n]||\text{ for } n\geq  {\tilde n}_0.
\end{equation}
We consider $\alpha_{n,s}$, $\beta_{n,s}$,  $\alpha_{n,s}'$, $\beta'_{n,s}$ given by \eqref{337}-\eqref{340} and in order to apply Lemma ~\ref{lem:61} we will check that 
\begin{equation} \label{425}
\rho_{n,4m}(j)\leq \rho_n(j) \text{ for } n\geq n_0+5m
\end{equation} 
holds with $\rho_n$ given by \eqref{47} and $\rho_{n,s}$ by \eqref{341}. Indeed, it is clear that 
\begin{align} \label{426}
\alpha_{n,s}(j)&\leq \alpha_s(j),\\
\label{427}
\beta_{n,s}(j)&\leq \beta_s(j)\coloneqq\frac{\alpha_s(j)}{\gamma_s(j)}\text{ for }j>n_0+m+s 
\end{align} 
hold with $\gamma_s$, $\alpha_s$ given by \eqref{43}, \eqref{44}. Then we observe that the function $\chi$ considered in Section~\ref{sec:41} can be chosen such that
\[
\norm{\chi'}_{\infty}\coloneqq\sup_{t\in\D{R}}\abs{\chi'(t)}\leq 2,
\]
where $\chi'$ denotes the derivative of $\chi$ and consequently
\[
\abs{\chi((j+1-n)/\tau_n)-\chi((j-n)/\tau_n)}\leq 2/{\tau_n},
\]
hence taking $\tilde\alpha_{n,s}$ as indicated in \eqref{46} we obtain 
\begin{equation} \label{428}
\alpha_{n,s}'(j)\leq\tilde\alpha_{n,s}(j)
\end{equation} 
due to \eqref{302}. Finally we need to estimate $\beta'_{n,s}(j)$ and for this purpose we first observe that 
\begin{align*}
j\geq n_0\ \&\ l>0&\implies d'_l(j)\geq d(j+1)-d(j)>0,\\
j\geq n_0-l\ \&\ l<0&\implies -d'_l(j)\geq d(j)-d(j-1)>0  
\end{align*}  
hold due to \eqref{40}. Further on we assume $j\geq n_0+m$. Thus for $|l|\leq m$ we have 
\begin{equation} \label{430}
\min \{ |d'_l(j)|,\, |d'_l(j+1)|\}\geq \gamma_0(j)  
\end{equation}
and writing $(\Delta d'_l)(j)=\sum_{i\in\C{I}(l)}\Delta^2d(j+i)$ where $\C{I}(l)=[0,\,l-1]\cap\D{Z}$ when $l>0$ and $\C{I}(l)=[l+1,\,0]\cap\D{Z}$ when $l<0$, we obtain
\begin{equation} \label{431}
|(\Delta d'_l)(j)|\leq |l|\tilde\gamma_l(j).
\end{equation}
Hence,
\begin{equation} \label{432}
\left| {\frac{1}{d'_l(j+1)}-\frac{1}{d'_l(j)}}\right|=\left|\frac{(\Delta d'_l)(j)}{d'_l(j+1)d'_l(j)}\right|
\leq\frac{|l|\tilde\gamma_l(j)}{\gamma_0(j)^2}.
\end{equation}
However by definition \eqref{423} we have 
\begin{equation} \label{433}
|\Delta p_{n,l}(j)|\leq\frac{|\Delta a_{n,l}(j)|}{|d'_l(j+1)|}+|a_{n,l}(j)|\,\left| {\Delta\left(\frac{1}{d'_l(j)}\right) }\right|,
\end{equation}
hence using $|l|\leq m$, \eqref{430} and \eqref{432} to estimate the right-hand side of \eqref{433} we obtain 
\begin{equation} \label{434}
\beta'_{n,s}(j)\leq \tilde{B}_{n,s}(j)\coloneqq\frac{\tilde\alpha_{n,s}(j)}{\gamma_s(j)} +\alpha_s(j)\,\frac{m\tilde\gamma_{m+s}(j)}{\gamma_s(j)^2} 
\end{equation}
for $n\geq n_0+m+s$. Now it is clear that \eqref{425} follows from \eqref{426}-\eqref{428} and \eqref{434}. Then using \eqref{425} in Lemma~\ref{lem:32} we obtain
\begin{equation} \label{436}
||[A_n,P_n]||\leq 10m^3\sup_{|i|\le 2\tau_n +4m}\rho_n(n+i)\text{ for } n\geq n_1.
\end{equation}
To complete the proof it remains to observe that one can use $\alpha_s$ and $\tilde\alpha_{n,s}$ as above if $\tilde{a}_{n,l}$ replaces $a_{n,l}$, hence the norm $||[\tilde A_n,P_n]||$ can be estimated in a similar manner. Thus we can conclude that the right-hand side of \eqref{424} can be estimated by the right-hand side of \eqref{48}.
\end{proof} 

\subsection*{Acknowledgments}

The second named author's research was partially supported by TODEQ MTKD-CT-2005-030042.


\begin{bibdiv}
\begin{biblist}
\bib{BNS}{article}{
   author={Boutet de Monvel, Anne},
   author={Naboko, Serguei},
   author={Silva, Luis O.},
   title={The asymptotic behavior of eigenvalues of a modified
   Jaynes-Cummings model},
   journal={Asymptot. Anal.},
   volume={47},
   date={2006},
   number={3-4},
   pages={291--315},
}
\bib{BZ}{article}{
   author={Boutet de Monvel, Anne},
   author={Zielinski, Lech},
   title={Explicit error estimates for eigenvalues of some unbounded Jacobi matrices},
   conference={
      title={Spectral Theory, Mathematical System Theory, Evolution Equations, Differential and Difference Equations: IWOTA10},
   },
   book={
      series={Oper. Theory Adv. Appl.},
      volume={221},
      publisher={Birkh\"auser Verlag},
      place={Basel},
   },
   date={2012},
   pages={187--215},
}
\bib{BN1}{article}{
   author={Behncke, H.},
   author={Nyamwala, F. Oluoch},
   title={Spectral theory of difference operators with almost constant
   coefficients},
   journal={J. Difference Equ. Appl.},
   volume={17},
   date={2011},
   number={5},
   pages={677--695},
}
\bib{BN2}{article}{
   author={Behncke, H.},
   author={Nyamwala, F. Oluoch},
   title={Spectral theory of difference operators with almost constant
   coefficients II},
   journal={J. Difference Equ. Appl.},
   volume={17},
   date={2011},
   number={5},
   pages={821--829},
}
\bib{JN}{article}{
   author={Janas, Jan},
   author={Naboko, Serguei},
   title={Infinite Jacobi matrices with unbounded entries: asymptotics of
   eigenvalues and the transformation operator approach},
   journal={SIAM J. Math. Anal.},
   volume={36},
   date={2004},
   number={2},
   pages={643--658},
}
\bib{M1}{article}{
   author={Malejki, Maria},
   title={Asymptotics of large eigenvalues for some discrete unbounded
   Jacobi matrices},
   journal={Linear Algebra Appl.},
   volume={431},
   date={2009},
   number={10},
   pages={1952--1970},
}
\bib{M2}{article}{
   author={Malejki, Maria},
   title={Asymptotic behaviour and approximation of eigenvalues for
   unbounded block Jacobi matrices},
   journal={Opuscula Math.},
   volume={30},
   date={2010},
   number={3},
   pages={311--330},
}
\bib{R-S}{book}{
   author={Reed, Michael},
   author={Simon, Barry},
   title={Methods of modern mathematical physics. II. Fourier analysis,
   self-adjointness},
   publisher={Academic Press [Harcourt Brace Jovanovich Publishers]},
   place={New York},
   date={1975},
   pages={xv+361},
}
\bib{Sh}{article}{
   author={Shubin, M. A.},
   title={Pseudodifference operators and their Green function},
   journal={Izv. Akad. Nauk SSSR Ser. Mat.},
   volume={49},
   date={1985},
   number={3},
   pages={652--671},
}
\bib{Tes}{book}{
   author={Teschl, Gerald},
   title={Jacobi operators and completely integrable nonlinear lattices},
   series={Mathematical Surveys and Monographs},
   volume={72},
   publisher={American Mathematical Society},
   place={Providence, RI},
   date={2000},
   pages={xvii+351},
}
\bib{Tur}{article}{
   author={Tur, {\`E}. A.},
   title={Jaynes--{C}ummings model: solution without rotating wave approximation},
   journal={Optics and Spectroscopy},
   volume={89},
   date={2000},
   number={4},
   pages={574--588},
}
\bib{V}{article}{
   author={Volkmer, Hans},
   title={Error estimates for Rayleigh-Ritz approximations of eigenvalues
   and eigenfunctions of the Mathieu and spheroidal wave equation},
   journal={Constr. Approx.},
   volume={20},
   date={2004},
   number={1},
   pages={39--54},
}
\bib{Z}{article}{
   author={Zielinski, Lech},
   title={Eigenvalue asymptotics for a class of Jacobi matrices},
   conference={
      title={Hot topics in operator theory, Conference Proceedings, Timi\c{s}oara, June 29-July 4, 2006},
   },
   book={
      series={Theta Ser. Adv. Math.},
      volume={9},
      publisher={Theta, Bucharest},
   },
   date={2008},
   pages={217--229},
}
\end{biblist}
\end{bibdiv}
\end{document}